\documentclass[a4paper, 10pt, conference]{ieeeconf}      
\IEEEoverridecommandlockouts                                                                                        
\usepackage{graphics} \usepackage{epsfig} \usepackage{mathptmx} \usepackage{times} \usepackage{amsmath} \usepackage{amssymb}     

\newtheorem{theorem}{Theorem}[section]

\newtheorem{prop}{Proposition}[section]
\newtheorem{definition}{Definition}[section]
\newtheorem{assum}{Assumption}[section]
\newtheorem{remark}{Remark}[section]

\title{\LARGE \bf
A mean-field game economic growth model
}

\author{Diogo Gomes$^{1}$, Laurent Lafleche$^{2}$, and  Levon Nurbekyan$^{3}$\thanks{$^{1}$D. Gomes is with the King Abdullah University of Science and Technology (KAUST), CEMSE Division , Thuwal 23955-6900. Saudi Arabia, and  
        KAUST SRI, Center for Uncertainty Quantification in Computational Science and Engineering.
        {\tt\small diogo.gomes@kaust.edu.sa}}\thanks{$^{2}$ Laurent Lafleche is with the King Abdullah University of Science and
Technology (KAUST), CEMSE Division , Thuwal 23955-6900. Saudi Arabia, and KAUST VSRP Program {\tt\small ll.laurent.lafleche@gmail.com} }\thanks{$^{3}$L. Nurbekyan is with the King Abdullah University of Science and
Technology (KAUST), CEMSE Division , Thuwal 23955-6900. Saudi Arabia.
        {\tt\small levon.nurbekyan@kaust.edu.sa} }
}

\def\ba{\textbf{a}}
\def\bc{\textbf{c}}

\def\bk{\textbf{k}}
\def\bp{\textbf{p}}
\def\bi{\textbf{i}}
\def\R{\mathbb{R}}
\def\L{\mathcal{L}}
\def\N{\mathbb{N}}
\newcommand             {\qa}           {\mathbf q_a}
\newcommand             {\qk}           {\mathbf q_k}

\begin{document}

\maketitle
\thispagestyle{empty}
\pagestyle{empty}

\begin{abstract}
Here, we examine a mean-field game (MFG) that models the economic growth of a population of non-cooperative rational agents. 
In this MFG, agents are described by two state variables - the capital and consumer goods they own. 
Each agent seeks to maximize their utility by taking into account statistical data of the total population. The individual actions drive the evolution of the players, and a market-clearing condition determines the relative price of capital and consumer goods.  We study the existence and uniqueness of optimal strategies of the agents and develop numerical methods
to compute these strategies and the equilibrium price. 
\end{abstract}

\section{INTRODUCTION}

Mathematical methods are central to modern economic theory and while
the behavior of economic agents cannot always be reduced to a precise mathematical formulation, utility maximization principles
and game-theoretical equilibria explain many important phenomena.
This is particularly relevant in problems where the hypothesis of rational expectations  holds \cite{lucas}. 

In numerous instances, the behavior of economic agents cannot be adequately captured by the representative agent assumption. Heterogeneous agent models
allow the study of questions where the differences between agents are of primary relevance, see \cite{bewley}, \cite{aiyagari}, or \cite{huggett}.

Mean-field games (MFG) were introduced in the engineering community in \cite{Caines1, Caines2}, and in the mathematical community in \cite{ll1,ll2,ll3} to model problems with a large number of agents acting non-cooperatively. Applications in economics and sustainable growth were some of earlier the motivations for these problems \cite{llg1, llg2}. 
While a substantial part of recent research on MFG has focused on 
mathematical questions \cite{porretta2,  GPM2, GPM3,
	MR3333058,  GPatVrt,  GPat, GS, cardaliaguet},
a number of emergent applications have been identified, 
including machine learning \cite{PQ1, PQ2}, price formation \cite{MR2835888}, 
\cite{MR2817378}, energy systems \cite{MR3030935, MR3195845}, and socio-economic applications \cite{MR3029461, GVW-Socio-Economic, GVW-dual}.
In economics, MFG  
provide a solid mathematical framework to 
investigate economic problems with heterogeneous agents \cite{lucasmoll, moll, moll2, moll3}. 

In \cite{GNP}, the authors present a mean-field game model
for economic growth in a population of rational agents. In this model, agents can produce and trade capital and consumer goods and seek to maximize a utility function.
An equilibrium condition determines the
 single macroeconomic variable, the price of capital goods measured in consumer goods.
From the mathematical point of view, little is known about this model. Here, we study the existence, uniqueness and number of other qualitative and quantitative properties of optimal strategies of agents.
We also develop numerical methods to calculate these optimal strategies and the equilibrium price. 

We close this introduction with a brief outline of the paper. We begin by developing the model in Section \ref{tm}. Next, in Section \ref{eot}, we investigate the optimal actions of the agents. A particular case with $N$ agents is examined in Section \ref{npl}. This $N$ player problem is then used to develop a numerical scheme, whose implementation as well as some results are presented in Section  \ref{sec: numerics}.
The paper ends,
in Section \ref{conc}, with a short discussion on open problems and future research directions.

\section{The model}
\label{tm}

The growth model we consider here was introduced in \cite{GNP}. This model features wealth (in the form of consumer goods) and capital accumulation
with a capital-dependent production function.

\subsection{Microeconomic framework}

At the time $t$, each agent has an amount  $\bk_t$  of capital and $\ba_t$ of consumer goods. 
Agents select a consumption rate $\bc_t$ and an 
investment rate  $\mathbf{e}_t$. 
The investment is made by exchanging consumer goods for capital goods. 
This exchange is made through a clearing mechanism at a price $\bp_t$
determined by the market. More precisely, agents exchange 
$\bp_t\mathbf{e}_t$ of consumer goods 
by an amount $\mathbf{e}_t$  of capital goods.

Agents use capital to produce either consumer goods or additional capital.
The production is specified by two production functions that depend
on the price level and capital. The production function for consumer goods is 
 $\Theta(k,p)$ and the production function for capital goods is $\Xi(k,p)$. 
Finally, our model also takes into account the possibility of capital depreciation through a function $g(k,p)$.

All the previous factors are assembled in the following microeconomic dynamics for a single agent:
\begin{equation}
\label{1}
\begin{cases}
\dot{\ba}_t =-\bc_t-\bp_t\mathbf{e}_t+\Theta(\bk_t,\bp_t)\\
\dot{\bk}_t= g(\bk_t,\bp_t)+\mathbf{e}_t+\Xi(\bk_t,\bp_t).
\end{cases}
\end{equation}
In alternative, we can regard the investment $\bi_t = \mathbf{e}_t+\Xi(\bk_t,\bp_t)$ as the control variable. Then,
we define the global production function
as
$F(k,p) := \Theta(\bk_t,\bp_t) + \bp_t \Xi(\bk_t,\bp_t)$. 
Accordingly, \eqref{1} becomes
\begin{equation}
\label{2}
\begin{cases}
\dot{\ba}_t = -\bc_t-\bp_t\bi_t+F(\bk_t,\bp_t)\\
\dot{\bk}_t = \bi_t + g(\bk_t,\bp_t).
\end{cases}
\end{equation}

Each agent has preferences on the value of his consumption, investment, consumer goods and  capital. These preferences are encoded in a utility function $u(a,k,c,i)$.
 We assume that $u$ is increasing and concave. The goal of each agent is to maximize his utility during a time horizon $T$ for a given initial amount of goods $\ba_0=a_0$ and capital $\bk_0=k_0$.
We define the utility or value function as
\begin{equation}\label{def_value_function}
V(a_0,k_0,t) := \sup_{(\bc,\bi)\in\mathcal{C}\times\mathcal{I}}\int_t^T u(\ba_s,\bk_s,\bc_s,\bi_s)d s,
\end{equation}
where $\mathcal{C}$ and $\mathcal{I}$ are the functional spaces on $(0,T)$ where the controls $\bi$ and $\bc$ are selected.

The Hamiltonian associated with the dynamics \eqref{2} and the optimization problem \eqref{def_value_function} is
\begin{align}
\label{def_H}
& H(a,k,q_a,q_k,p) :=\\ 
&\sup_{(c,i)\in\R^2}((-c-pi+F(k,p))q_a+(i+g(k,p))q_k+u(a,k, c,i)).
\end{align}
If the value function $V$ in \eqref{def_value_function} is continuously differentiable,  then it solves the Hamilton-Jacobi equation
\begin{equation}\label{eq_HJ}
\partial_tV(a,k,t) + H(a,k,\partial_a V(a,k,t),\partial_k V(a,k,t),\bp_t) = 0,
\end{equation}
with $V(a,k,T) = 0$.
Moreover, any twice continuously differentiable solution of the previous equation is the value function.

Finally, given a twice continuously differentiable solution of \eqref{eq_HJ}, 
the optimal controls $c_t^*(a,k)$ and $i_t^*(a,k)$ 
are determined in feedback form through the equations
\begin{equation}
\begin{cases}
\partial_{q_a}H(a,k,\partial_aV,\partial_kV,\bp_t)  =  -c_t^* - \bp_t i_t^* + F(k,\bp_t)\\
\partial_{q_k}H(a,k,\partial_aV,\partial_kV,\bp_t)  =  i_t^* + g(k,\bp_t).
\end{cases}
\end{equation}

\subsection{Macroeconomic setting}

On the macroeconomic scale, agents are described by a probability density
$\rho$; 
for each time  $t\in[0,T]$, $\rho_t(a,k)$ is the probability density  of agents with an amount $a$ of consumer goods and $k$ of capital. Consequently, 
\[
\int_{\R^2}\rho_t(a,k) d a d k = 1, \quad \text{for}\ t\in[0,T].
\]
We assume all agents are rational. So, the dynamics of each individual is given by 
\[
\left\lbrace
\begin{array}{rcl}
\dot{\ba}_t & = & \partial_{q_a}H(\ba_t,\bk_t,\partial_aV(\ba_t,\bk_t,t),\partial_kV(\ba_t,\bk_t,t),\bp_t)\\
\dot{\bk}_t & = & \partial_{q_k}H(\ba_t,\bk_t,\partial_aV(\ba_t,\bk_t,t),\partial_kV(\ba_t,\bk_t,t),\bp_t).\\
\end{array}
\right.
\]
Accordingly, the density of players solves the transport equation
\begin{equation}\label{eq_transport}
\partial_t\rho + \partial_a(\partial_{q_a}H\rho) + \partial_k(\partial_{q_k}H\rho) = 0.
\end{equation}

\subsection{Mean-field game formulation of the problem}

As the agents can only exchange consumer goods and capital with other agents, we 
have a balance equation 
\[
\int_{\R^2}e_t^*(a,k)\rho(a,k,t) d a d k = 0,
\]
for $0\leq t\leq T$. 
Equivalently,
\begin{equation}\label{eq_equilibre_p}
\int_{\R^2}i_t^*(a,k)\rho_t(a,k) d a d k = \int_{\R^2}\Xi(k,\bp_t)\rho_t(a,k) d a d k. 
\end{equation}
The previous balance condition, the Hamilton-Jacobi equation (\ref{eq_HJ}), and the transport equation (\ref{eq_transport}) determine a mean-field game formulation of the problem
\begin{equation}\label{syst_mfg}
\begin{cases}
\partial_tV  +  H(a,k,\partial_aV,\partial_kV,\bp_t)  =  0\\
\partial_t\rho  +  \partial_a(\partial_{q_a}H\rho) + \partial_k(\partial_{q_k}H\rho)  =  0
\end{cases}
\end{equation}
with
\[
V(a,k,T) = 0,\quad  \rho(a,k,0)=\rho_0(a,k)
\]
for some $\rho_0(a,k)$.
The coupling between these two equations
follows from the equilibrium equation (\ref{eq_equilibre_p})
that
determines
$\bp_t$.

\section{Existence of optimal trajectories}
\label{eot}

In this section, we  discuss assumptions that ensure that problem  (\ref{def_value_function}) is well posed.  We are interested in the existence and uniqueness of optimal trajectories and controls under a given price dynamics. If the utility function $u$ is concave, and the controls $c$ and $i$ are bounded, then  (\ref{def_value_function}) admits a maximizer (see \cite{GNP}). Here, we do not  assume a priori that the controls are bounded. Instead, we impose natural growth constraints on the utility function $u$ and prove the existence of optimal trajectories and controls.

In our analysis, in  contrast to classical optimal control theory problems, the utility function $u$ is not coercive. Hence, a central technical point that we need to address is the boundedness of the controls in the interval $(0,T)$.

The uniqueness of optimal trajectories for strictly concave utility functions was proven in \cite{GNP}.
 
\subsection{Main assumptions}

Our main assumptions are as follows.

                \begin{assum}[g is sub-linear]\label{gsublinear}
                        There exists $(g_1,g_2) \in\L^\infty_{loc}(\R)^2$ such that
                        \[
                        g(k,p)\leq g_1(p)k+g_2(p),
                        \]
                        for all $(k,p)\in\R^2$.

                        \end{assum}

The previous assumption includes, for instance the important case of a linear depreciation function $g(k,p)=-\beta k$, for some $\beta>0$.

                        \begin{assum}[F is sub-linear]\label{fsublinear}
                                There exists $(F_1,F_2) \in\L^\infty_{loc}(\R)^2$ such that 
                                \[
                                F(k,p)\leq F_1(p)k+F_2(p),
                                \]
                                for all $(k,p)\in\R^2$. 
                                \end{assum}
                                
                                Before stating the next assumption, we recall that for $x\in \R$, $x^+=\max\{x,0\}$, and $x^-=\max\{-x,0\}$.

                                \begin{assum}
                                        \label{gammaalpha}
                                        There exists $(\gamma_j)_{j\in\{ 1,4\}}\in[0,1]^4$, $ (\alpha_j)_{j\in\{ 1,4\}}\in[1,+\infty[^4$ and $C_u\in\R_+^*$ such that
                                        \begin{align*}
                                                u(a, k, c,i) &\leq C_u ( 1 + (c^+)^{\gamma_1} + (a^+)^{\gamma_3} + (k^+)^{\gamma_4}\\
                                                &- (c^-)^{\alpha_1} - |i|^{\alpha_2} - (a^-)^{\alpha_3} - (k^-)^{\alpha_4}),
                                        \end{align*}
                                        for all $(c,i,a,k)\in\R^4$.
                                        \end{assum}

                                                        \begin{definition}\label{def_loc_L1}
                                                                Let $A,B$ and $C$ be normed vector spaces, $\Omega\times I\subset A\times B$ be open subsets. 
                                                                $F : \Omega\times I \to C$ is \textbf{locally uniformly integrable} in the second variable if for all $(x_0,t_0)\in\Omega\times I$ there exists $X\times J$ neighborhood of $(x_0,t_0)$ so that
                                                                \[
                                                                \left(t\mapsto \sup_{x\in X}|F(x,t)|\right)\in L^1(J).
                                                                \]
                                                                \end{definition}

\subsection{Existence of Solutions}

Finally, we state our main result on the existence of solutions. \begin{theorem}\label{th_existence}
Let $L^1_{R,loc}:=\{c\in L^1_{loc}\left([0,T)\right), \int_0^Tc^+\leq R\}$.\\
Assume $u\in C^1(\R^4)$ is concave, verifies Assumption \ref{gammaalpha}, and that there exists $C\in\R^+$ such that
\[
\left|\frac{\partial u}{\partial x_i}(x_1,...x_4)\right| \leq C\left(1+\sum_{j=1}^{4}|x_j|^{\alpha_j}\right)
\]
for all $1\leq i\leq 4$. Moreover, suppose that $\bp$ is bounded in $(0,T)$, $F$ and $g$ are sub-linear (Assumptions \ref{gsublinear} and  \ref{fsublinear}) and such that
\begin{itemize}
        \item $\forall k\in\R,\ g(k,\cdot)\in\L^\infty_{loc}(\R)$,
        \item $\forall k\in\R,\ F(k,\cdot)\in\L^\infty_{loc}(\R)$,
        \item $(k,t)\mapsto F(k,\bp_t)$ is locally uniformly integrable in the second variable.
\end{itemize}
Let $P:=\|\bp\|_{L^\infty(0,T)}$ and $R\in\R$ be such that $R \geq P \|g(k_0,\cdot)\|_{\L^\infty(-P,P)} + \|F(k_0,\cdot)\|_{\L^\infty(-P,P)}$. Then, for a given initial quantity $(a_0,k_0)\in\R^2$ of consumer goods and capital, there exists a maximizer to (\ref{def_value_function}) with $\mathcal{C}=L^1_{R,loc}$ and $\mathcal{I}=L^1_{loc}$.
\end{theorem}
\begin{remark} If, for each $p$, the functions $k\mapsto F(k,p)$ and $k\mapsto g(k,p)$ are concave, the hypotheses of Theorem \ref{th_existence} hold, and we obtain the existence of optimal trajectories. 
\end{remark}
\begin{proof}
We sketch here the proof of the theorem that is based on 
the direct method in the calculus of variations. 
We select a minimizing sequence $(\ba_t^k, \bk_t^k, \bc^k_t, \bi_t^k)$. Assumption 
\ref{gammaalpha} ensures the boundedness of this sequence.
Therefore, we can extract a weakly convergent subsequence. 
Next, Assumptions
\ref{gsublinear} and \ref{fsublinear} give that the limit is a admissible trajectory for \eqref{2}. Finally, the concavity of $u$ implies that this trajectory is a minimizer by weak upper semicontinuity of \eqref{def_value_function}.
\end{proof}

\section{$N$-player approximation}
\label{npl}

In this section, we construct a $N$-player approximation of our growth model. We assume that the population in our model consists of a fixed number of players $N\in\N$ and obtain a system of equations for the optimal trajectories.

Let $\ba^n_t,\bk^n_t$ be the amount of consumer goods and capital of the player $n$ at time $t$, for $1\leq n \leq N$. We select an individual agent. As shown in the previous section, given a price dynamics $\bp_t$, we can always find an optimal strategy for this player. In the finite player case, we assume arbitrary initial wealth in consumer goods and capital for the $N$ players. Then, using individual optimal strategies, we build an explicit solution to the transport equation (\ref{eq_transport}). Hence, here, we provide a theoretical basis for our numerical computations.

A full analysis of the mean-field game \eqref{syst_mfg} requires the computation of the equilibrium price. 
This price is determined by the equilibrium condition \eqref{eq_equilibre_p}. We calculate this price using a fixed point formulation for the equilibrium condition (see Section \ref{sec: numerics}).

\subsection{Solution to the Transport Equation}

Let $a^n_0 = \ba^n_0$ and $k^n_0 = \bk^n_0$ be initial consumer goods and capital levels for the $N$\ agents. On the macroeconomic scale, the density of agents at the initial time $t=0$ is the probability measure
\[
\rho_0 = \dfrac{1}{N}\sum_{n=1}^{N} \delta_{(a^n_0,k^n_0)},
\]
where $\delta$ is the Dirac delta.

We recall that a probability measure $\rho$ solves \eqref{eq_transport}
in the sense of distributions if for every $\varphi\in C^\infty_c(\R^2\times \R_+)$, we have
\begin{equation}
\label{ssd}
\int_0^{+\infty}
\int_{\R^2}
(\partial_t\varphi  +  \partial_a\varphi\partial_{q_a}H + \partial_k\varphi \partial_{q_k}H)\rho dk da dt=0.
\end{equation}

\begin{prop}\label{prop_solutiontransport}
        Suppose $(\ba^n,\bk^n)\in C^0([0,T])^2$ solve \eqref{2} with initial data $(a_0^n,k_0^n)$, respectively. Then
        \begin{equation}
        \label{rdef}
        \rho_t := \dfrac{1}{N}\sum_{n=1}^{N} \delta_{(\ba^n_t,\bk^n_t)}
        \end{equation}
        solves the transport equation \eqref{eq_transport} in the sense of distributions and verifies
        \[
         \int_{\R^2}\rho_t(a,k) dadk = 1,\ \text{for all}\ t\in[0,T].
        \]
\end{prop}
\begin{proof}
The proof follows by combining \eqref{rdef} with \eqref{ssd}.	
\end{proof}

\subsection{Computation of the Hamiltonian}\label{subsection_computeH}

Define $\qa^n(t) := \partial_aV(\ba^n_t,\bk^n_t,t)$ and $\qk^n(t) := \partial_kV(\ba^n_t,\bk^n_t,t)$, where $V$ is given by \eqref{def_value_function}. As $V(\cdot,\cdot,T) = 0$, we have $\qa^n(T) = 0$ and $\qk^n(T)=0$.
Set $h(a,k,c,i,q_a,q_k,p) := (-c-pi+F(k,p))q_a+(i+g(k,p))q_k+u(a,k,c,i)$ and assume $u\in C^1(\R^4)$.

The Hamiltonian is 
\[
H(a,k,q_a,q_k,p) := \sup_{(c,i)\in\R^2}h(a,k,c,i,q_a,q_k,p).
\]
For the optimal $c$ and $i$, we have $\partial_ch=0$ and $\partial_ih=0$.
Hence
\begin{equation}\label{syst_dh=0}
\left\lbrace
\begin{array}{rcl}
q_a & = & \partial_cu(a,k,c^*,i^*)\\
q_k & = & p\partial_cu(a,k,c^*,i^*) - \partial_iu(a,k,c^*,i^*).\\
\end{array}
\right.
\end{equation}
If $u$ is strictly concave in $c$ and $\partial_cu(a,k,c,i)=u_c(c)$ depends only on $c$, then $u_c$ is strictly decreasing and we can compute $c^* = u_c^{-1}(q_a)$. The equation for $q_k$ can be written $\partial_iu(a,k,c^*,i^*)=pq_a-q_k$. As for $c$, if $\partial_iu(a,k,c,i)=u_i(i)$ depends only on $i$ is strictly decreasing, then $i^* = u_i^{-1}(pq_a-q_k)$. Consequently, we obtain
\begin{equation}
H(a,k,q_a,q_k,p) = h\left(a,k,u_c^{-1}(q_a),u_i^{-1}(pq_a-q_k),q_a,q_k,p\right).
\end{equation}

\subsection{Computation of the optimal trajectories}

        Assume $V$ is $C^1$. Then $V$ solves the Hamilton-Jacobi equation \eqref{eq_HJ}. Moreover, the method of characteristics  gives the Hamiltonian system \cite{GNP}
        \begin{equation}\label{syst_H}
        \left\lbrace
        \begin{array}{rcl}
        \dot{\ba}^n_t & = &
        \partial_{q_a}H(\ba^n_t,\bk^n_t,\qa^n(t),\qk^n(t),\bp_t)\\
        \dot{\bk}^n_t & = &
        \partial_{q_k}H(\ba^n_t,\bk^n_t,\qa^n(t),\qk^n(t),\bp_t)\\
        \frac{d}{d t}\qa^n(t) & = &
        - \partial_{a}H(\ba^n_t,\bk^n_t,\qa^n(t),\qk^n(t),\bp_t)\\
        \frac{d}{d t}\qk^n(t) & = &
        - \partial_{k}H(\ba^n_t,\bk^n_t,\qa^n(t),\qk^n(t),\bp_t)\\
        \end{array}
        \right.
        \end{equation}for the optimal trajectories of the $N$\ players. This system has initial/terminal conditions given by
        \begin{equation}
        \label{14}
        \ba^n_0=a^n_0,\ \bk^n_0=k^n_0,\  \qa^n(T) = 0,\ \text{and}\ \qk^n(T)=0.
        \end{equation}
        
        Reciprocally, we have the following proposition whose proof is given  in \cite{GNP}:
        \begin{prop}\label{prop_sufficient_condition}
                Assume that the utility function $u\in C^1(\R^4)$ is concave and is non-decreasing in $a$ and $k$, that the production function, $F$, is non-decreasing and concave in $k$ and that the depreciation function, $g$, is concave in $k$.\\
                If $(\ba,\bk,\qa,\qk)$ solves \eqref{syst_H} with the initial/terminal conditions \eqref{14}, then $\ba$ and $\bk$ are optimal trajectories for \eqref{def_value_function} and
                \[
                \qa(t)\geq 0\ \text{and}\ \qk(t) \geq 0\ \text{for all}\ t\in[0,T].
                \]
        \end{prop}
        \begin{proof}
        See \cite{GNP}.
      	\end{proof}

        Taking into account Proposition \ref{prop_sufficient_condition}, we aim to solve the equation \eqref{syst_H} with corresponding initial/terminal conditions. 
        
        Firstly, we solve \eqref{syst_H} with initial conditions $(a_0^n,k^0_n)$ and $(\qa^n(0),\qk^n(0))$. The corresponding solution generates a function $Q^n : (\qa^n(0),\qk^n(0)) \mapsto (\qa^n(T),\qk^n(T))$.
        Next, we determine the initial conditions $ (\qa^n(0),\qk^n(0))$
such that $Q^n(q)=0$. In general, this inversion cannot be done analytically. 

\section{Numerical results}\label{sec: numerics}

Here, we outline our algorithm, discuss a model problem and present the corresponding numerical results.

\subsection{Description of the algorithm}

To compute the solutions to the mean-field game, we proceed as follows. 
We consider $N$ agents with initial state $(a_0^n,k_0^n)$, $1\leq n\leq N$. 
We
discretize the price function $\bp_t$ with $m+1$ equidistant points $
\bp_{\frac i m T}$, $0\leq i\leq m$. To recover the price at arbitrary times $t$, we use a third-order interpolant.

We define $q_{a_0}(a_0,k_0)$ and $q_{k_0}(a_0, k_0)$
to be the values for which the solution of \eqref{syst_H} with the initial condition
$(a_0, k_0,q_{a_0}(a_0,k_0),q_{k_0}(a_0, k_0) )$
satisfies $\qa(T)=\qk(T)=0$. This function is determined by solving numerically this system of two equations. 
For a fixed price $\bp_t$, we denote $(\ba^n(t),\bk^n(t))$, $1\leq n\leq N$
to be the solution of \eqref{syst_H} with initial condition 
$(a_0^n, k_0^n,q_{a_0}(a_0^n,k_0^n),q_{k_0}(a_0^n, k_0^n))$.

Next, we examine the equilibrium condition. 
For a given price $\bp_t$, the imbalance function 
is
\begin{align*}
\iota(t; \bp)=\frac{1}{N}
\sum_{n=1}^N
&\Big[\partial_{q_k}H(\ba^n(t), \bk^n(t), \qa^n(t), \qk^n(t),\bp_t)
\\
& -g(\bk^n(t), \bp_t)- \Xi(\bk^n(t), \bp_t)\Big].
\end{align*}
To find $\bp$, we look at a fixed point of the function
\begin{equation}
\label{psi}
\bp_t\mapsto \bp_t)+\mu\iota(t,\bp)\equiv \Psi(\bp)
\end{equation}
for some $\mu> 0$. Clearly, a fixed point of $\Psi$ is 
an equilibrium price. The construction of $\Psi$ corresponds to
the intuition that when supply does not meet demand prices should increase
whereas if supply exceeds demand, prices should decrease. 
When $\iota>0$, demand exceeds supply. Hence, prices are too low and are increased. In contrast, if $\iota<0$
supply exceeds demand. Accordingly, 
prices that are too high and are decreased. 

\subsection{A model problem}

To illustrate our methods, we consider a concrete example where
consumers produce consumer goods and capital
at a rate that depends linearly on the capital. 
We set
\[
\Theta(k,p)=k, \qquad \Xi(k,p)=0.1 k, 
\]
and so
\[
F(k,p)=k (0.1 p+1).
\]
We assume that the depreciation function is $g(k,p)=-\frac k 2$.
Next, we define
\[
u_1(x)=\begin{cases}
\sqrt{x+\frac{1}{16}}\qquad x>0,\\
\frac{5}{4}-(1-x)^2\qquad x\leq 0. 
\end{cases}
\]
We note that $u_1$ is a monotone increasing function of class $C^1$. 

Next, we consider the utility function
\[
u(a,k,c,i)=u_1(c)+u_1(a)+u_1(k)-\frac{i^2}2.
\]
This utility function satisfies the hypothesis in Section \ref{eot} and 
the Hamiltonian is
$
H = 
H^A + H^B + H^C, 
$
where
\[
H^A(a, k, q_a, q_k, p)=F(k, p) q_a + g(k, p) q_k + u_1(a) + u_1(k),
\]
\begin{align*}
H^B(a, k, q_a, q_k, p)&=
\sup_i \left[(-p q_a+q_k) i-\frac{i^2}2\right]
\\
&=(- p q_a + q_k)^2/2,
\end{align*}
and
\[
H^C(a, k, q_a, q_k, p)=
-q_a c^*(q_a) + u_1(c^*(q_a)), 
\]
with
\[
c^*(q_a)=
\begin{cases}
\frac{4-\text{qa}^2}{16 \text{qa}^2}\qquad q_a>2\\
\frac{2-\text{qa}}{2}\qquad q_a\leq 2.
\end{cases}
\]

\subsection{Numerical results} 
 
We considered $25$ agents, with initial conditions equally spaced
in $\left[\frac 1 2, \frac 3 2\right]\times \left[\frac 1 2, \frac 3 2\right]$;
the terminal time is $T=1$, and $\mu=0.8$ in \eqref{psi} . Our code was implemented in Mathematica, and we used the built-in  solver NSolve. 

We began with the initial $\bp^0_t\equiv 1$. 
Figure \ref{pricefig} depicts the equilibrium price (achieved after 10 iterations) and Figure \ref{trajfig} shows the trajectories of the agents. 
Finally, the average consumer goods and capital as a function of time are shown in Figures \ref{consumerfig} and \ref{capitalfig}, respectively. 
Agents to accumulate consumer goods in the beginning and consume more in later times, as can be seen in Figure \ref{consumerfig}.
Due to the balance condition, 
the average capital $\bar \bk$ solves the ODE
\[
\dot {\bar \bk}=-0.4 \bar \bk.
\]
In Figure  \ref{capitalfig}, we see the exact solution (dashed) superimposed on the numerical solution. Finally, we observe in Figure \ref{pricefig} that the price is a monotone decreasing function of time.

   \begin{figure}[thpb]
      \begin{center}
      \includegraphics[scale=0.5]{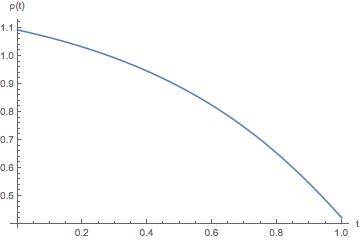}
      \caption{Price}
      \label{pricefig}
      \end{center}
   \end{figure}

 \begin{figure}[thpb]
        \begin{center}
        \includegraphics[scale=0.5]{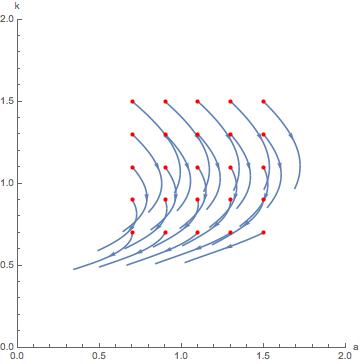}
        \caption{Agent's trajectories (initial conditions in red).}
        \label{trajfig}
        \end{center}
 \end{figure}

\begin{figure}[thpb]
        \begin{center}
                \includegraphics[scale=0.5]{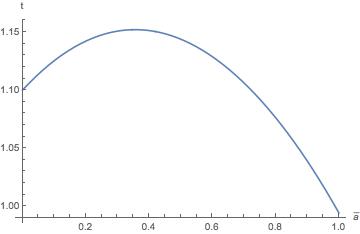}
                \caption{Average consumer goods - $\bar a$.}
                \label{consumerfig}
        \end{center}
\end{figure}

\begin{figure}[thpb]
        \begin{center}
                \includegraphics[scale=0.5]{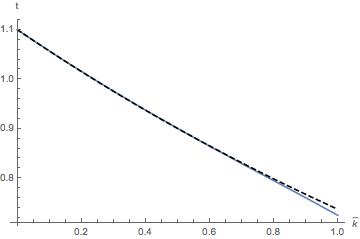}
                \caption{Average capital - $\bar k$ (exact solution dashed).}
                \label{capitalfig}
        \end{center}
\end{figure}

\section{Conclusions and further work}
\label{conc}

Here, we have developed the qualitative theory, approximation methods and 
a novel method for the computation of equilibrium prices in the economic growth model introduced in \cite{GNP}. Our simulations illustrate 
some of the qualitative properties that we would expect such a model to satisfy.
Even so, numerous questions remain open. For instance, what is the influence 
on the results
of the choice of production functions $\Theta$ and $\Xi$?
Here, the production functions are price-independent, but this may not be the case in a more realistic model.
Is the price a decreasing function?
We expect so, at least close to the terminal time since near 
$T$ consumption should take priority over investment. 
Can we prove convergence 
of the numerical method? While our iterative algorithm is based on sound economic principles, we would like to see a mathematical proof of convergence. Finally, in limited number of simulations the price seems to be unique. However, this is another issue that is still open.  We believe that 
these are important questions that should
be examined in the future.

\addtolength{\textheight}{-12cm}

\section*{Acknowledgment}

This work was partially supported by KAUST baseline and start-up funds, KAUST SRI, Uncertainty Quantification Center in Computational Science and Engineering and KAUST VSRP program.

\end{document}